\documentclass[12pt]{amsart}
\usepackage{amsmath,amssymb}
\title[Mankiewicz's Theorem and the Mazur--Ulam property]{Mankiewicz's 
theorem and the Mazur--Ulam property for $\mathrm{C}^*$-alge\-bras}
\author{Michiya Mori}
\address{Graduate School of Mathematical Sciences, The University of Tokyo, \mbox{153-8914} Japan}
\email{mmori@ms.u-tokyo.ac.jp}
\author{Narutaka Ozawa}
\address{RIMS, Kyoto University, \mbox{606-8502} Japan}
\email{narutaka@kurims.kyoto-u.ac.jp}
\subjclass{46B20; 46B04, 46L05}

\keywords{Mankiewicz's theorem, Mazur--Ulam property, Tingley's problem, $\mathrm{C}^*$-alge\-bras}
\date{\today}

\newtheorem{thm}{Theorem}
\newtheorem{prop}[thm]{Proposition}
\newtheorem{cor}[thm]{Corollary}
\newtheorem{lem}[thm]{Lemma}
\theoremstyle{definition}

\newtheorem{example}[thm]{Example}

\newcommand{\IB}{\mathbb B}
\newcommand{\IC}{\mathbb C}

\newcommand{\IF}{\mathbb F}
\newcommand{\IH}{\mathbb H}

\newcommand{\IM}{\mathbb M}
\newcommand{\IN}{\mathbb N}
\newcommand{\IR}{\mathbb R}

\newcommand{\cD}{\mathcal D}
\newcommand{\cH}{\mathcal H}
\newcommand{\cJ}{\mathcal J}
\newcommand{\cU}{\mathcal U}
\newcommand{\cN}{\mathcal N}

\newcommand{\vp}{\varphi}
\newcommand{\ve}{\varepsilon}

\DeclareMathOperator{\tr}{tr}
\DeclareMathOperator{\id}{id}
\DeclareMathOperator{\ext}{ext}

\DeclareMathOperator{\dist}{dist}
\DeclareMathOperator{\conv}{conv}

\DeclareMathOperator{\supp}{supp}

\newcommand{\ip}[1]{\mathopen{\langle}#1\mathclose{\rangle}}
\begin{document}
\begin{abstract}
We prove that every unital $\mathrm{C}^*$-alge\-bra $A$ 
has the Mazur--Ulam property. Namely, 
every surjective isometry from the unit sphere $S_A$ of $A$ 
onto the unit sphere $S_Y$ of another normed space $Y$ extends 
to a real linear map. 
This extends the result of A.~M.~Peralta and F.~J.~Fern\'andez-Polo 
who have proved the same under the additional assumption that both $A$ and $Y$ 
are von Neumann algebras. In the course of the proof, 
we strengthen Mankiewicz's theorem and prove that 
every surjective isometry from a closed unit ball with enough extreme 
points onto an arbitrary convex subset of a normed space is necessarily affine. 
\end{abstract}
\maketitle
\section{Introduction}
The celebrated Mazur--Ulam theorem (\cite{mu}) asserts that 
every surjective isometry between normed spaces $X$ and $Y$ is necessarily affine. 
This was extended by P.~Mankiewicz (\cite{mankiewicz}) to surjective isometries 
between the closed unit balls $B_X$ and $B_Y$. 
Motivated by these results, D.~Tingley (\cite{tingley}) posed the following problem in 1987: 
Does every surjective isometry $T\colon S_X \to S_Y$ between 
the unit spheres of normed spaces $X$ and $Y$ 
extend to a real linear isometry between $X$ and $Y$? 
Currently, no counterexample to Tingley's problem is known. 
A Banach space $X$ is said to have the \emph{Mazur--Ulam property} (\cite{cd}) if 
Tingley's problem has an affirmative answer for an arbitrary target $Y$. 
The main result of the present paper is the following. 
\begin{thm}\label{thm:mau}
Every unital complex $\mathrm{C}^*$-alge\-bra (as a real Banach space) 
and every real von Neumann algebra has the Mazur--Ulam property.
\end{thm}
Tingley's problem between von Neumann algebras has been solved earlier 
in \cite{pfp} (see also \cite{mori,tanaka2}).
See \cite{peralta} and the references herein 
for more information about Tingley's problem for Banach spaces related 
to operator algebras. 
The Mazur--Ulam property for commutative $\mathrm{C}^*$-alge\-bras 
has been proved in \cite{ding,fw,liu,pca}.
For more examples of Banach spaces with 
the Mazur--Ulam property, see \cite{thl} for example. 
The starting point of the present and many other works on Tingley's problem 
for operator algebras is R.~Tanaka's observation (\cite{tanaka2}) that 
a surjective isometry from the unit sphere $S_A$ of 
a $\mathrm{C}^*$-alge\-bra $A$ onto another unit sphere $S_Y$ 
maps closed faces onto closed faces. 

The unit sphere of a $\mathrm{C}^*$-alge\-bra has many faces 
that are approximable by closed unit balls of $\mathrm{C}^*$-alge\^bras.
To exploit this, 
we revisit Mankiewicz's theorem. 
Let $K$ be a convex subset in a normed space $X$. 
The subset $K$ is called a \emph{convex body} if it has non-empty interior in $X$. 
Mankiewicz (\cite{mankiewicz}) has proved that any isometry between 
convex bodies is necessarily affine. 
We say $K$ has the \emph{strong Mankiewicz property} if every surjective 
isometry $T$ from $K$ onto an arbitrary convex subset $L$ in a normed 
space $Y$ is affine. 
Every convex subset of a strictly convex normed space has this property 
because it is uniquely geodesic (cf.\ Lemma 6.1 in \cite{bfgm}), but 
some convex subset of $L^1[0,1]$ does not (\cite{schechtman}), see Example~\ref{example}. 
Every normed space also has this property  
by Figiel's theorem (\cite{figiel}). 
This probably suggests that the same is true for every convex body, 
but this is not clear to the authors 
(since the range $L$ is not assumed to have non-empty interior).
\begin{thm}\label{thm:man}
Let $X$ be a Banach space such that 
the closed convex hull of the extreme points $\ext B_X$ 
of the closed unit ball $B_X$ has non-empty interior in $X$. 
Then, every convex body $K\subset X$ has the strong Mankiewicz property.
\end{thm}
\begin{cor}\label{cor:man}
Let $A$ be a unital complex $\mathrm{C}^*$-alge\-bra or 
a real von Neumann algebra. 
Then, every convex body in $A$ has the strong Mankiewicz property. 
\end{cor}

\subsection*{Acknowledgment}
We are grateful to Professor G.~Schechtman for allowing us to include his example 
from \cite{schechtman}. 
The first author is partially supported by Leading Graduate Course for 
Frontiers of Mathematical Sciences and Physics, MEXT, Japan. 
The second author is partially supported by JSPS KAKENHI 17K05277 and 15H05739.
\subsection*{Notations and Remarks on real $\mathrm{C}^*$-alge\-bras}
Throughout this paper, $X$ and $Y$ are real normed (Banach) spaces and 
$A$ is a $\mathrm{C}^*$-alge\-bra which may be real or complex unless otherwise stated.
The unit sphere and the closed unit ball of $X$ are denoted respectively 
by $S_X$ and $B_X$. 
For any projection $p$ in a unital $\mathrm{C}^*$-alge\-bra, 
we write $p^\perp:=1-p$.
By definition, a real $\mathrm{C}^*$-alge\-bra $A$ 
is the real part $\{ a \in A_{\IC} : \cJ(a)=a\}$ 
of a complex $\mathrm{C}^*$-alge\-bra $A_{\IC}$ with respect to  
a conjugate-linear $*$-auto\-mor\-phism $\cJ$ such that $\cJ^2=\id$. 
Many of the standard operations in complex $\mathrm{C}^*$-alge\-bras 
work equally well for real $\mathrm{C}^*$-alge\-bras. 
For example, the modulus $|a|$ of an element $a$ in $A$ 
is firstly considered in the complexification $A_{\IC}$ 
and by uniqueness one sees that $|a|$ belongs to the real part $A$.
For any complex continuous function $f$ and any normal 
element $a$ in $A_{\IC}$, one has $\cJ(f(a))=\bar{f}(\cJ(a^*))$. 
A projection $p$ in a real von Neumann algebra $A$ is minimal 
if and only if $pAp=\IR,\,\IC,\,\IH$, in which case $p$ has rank 
at most $2$ in the complexification $A_{\IC}$. 
See \cite{li} for more on real operator algebras.
\section{On the strong Mankiewicz property}
\begin{lem}
If $B_X$ has the strong Mankiewicz property, then every convex body $K\subset X$ has 
the strong Mankiewicz property.
\end{lem}
\begin{proof}
Let a surjective isometry $T$ be given.
The assumption implies that every interior point $x$ in $K$ 
has a neighborhood on which $T$ is affine. 
By continuation (see Proof of Theorem~2 in \cite{mankiewicz}), 
one sees that $T$ is affine everywhere. 
\end{proof}

\begin{proof}[Proof of Theorem~\ref{thm:man}]
By the above lemma, it suffices to show 
every surjective isometry $T\colon B_X\to L\subset Y$ is affine. 
We may assume that $T(0)=0$, which implies that 
$\|T(x)\|=\|x\|$ for all $x\in B_X$.
Let $a\in\ext B_X$. Since the line segment $[-a,a]$ is the unique geodesic path 
between $-a$ and $a$, the map $T$ is affine (linear) on $[-a,a]$. 
We claim that if $x\in X$ and $\lambda\in\IR$ are such that 
$\|x\|\le\frac{1}{2}$ and $\|x\|+|\lambda|\le1$, then
\[
\|T(x+\lambda a) - (T(x)+\lambda T(a))\|\le 4\|x\| |\lambda|.
\] 
For this, we may assume that $\lambda\geq0$ as $T(-a)=-T(a)$.
Since $T$ is affine on $[x,x+(1-\|x\|) a]$ by the similar reason as above, 
one has 
\begin{align*}
T(x+\lambda a) - T(x) = \frac{\lambda}{1-\|x\|}(T(x+(1-\|x\|) a)-T(x))
 \approx_{4\|x\|\lambda } \lambda T(a).
\end{align*}
Let $T_n\colon n B_X\to nB_Y$ be the map defined by 
$T_n(x)=nT(\frac{1}{n}x)$. 
By the previous inequality, for any $a_k\in\ext B_X$ and $\lambda_k\in\IR$ 
such that $C:=\sum_k|\lambda_k|$, one has 
\[\tag{$\ast$}
\| T_n(\sum_k\lambda_k a_k) - \sum_k \lambda_k T(a_k) \| \le 4n(\sum_k\frac{|\lambda_k|}{n})^2
\le \frac{4C^2}{n}
\]
for all $n\geq 2C$. 
We consider the Banach space $Z_\infty:=\ell_\infty(\IN;Z)/c_0(\IN;Z)$ for $Z=X$ or $Y$ 
and define $\hat{T}\colon X_\infty\to Y_\infty$ by 
$\hat{T}([x_n]_n)=[T_n(x_n)]_n$. 
Here $[x_n]_n$ denotes the element in $X_\infty$ represented 
by $(x_n)_n \in \ell_\infty(\IN;X)$. 
Observe that $\hat{T}$ is a well-defined isometry 
which is moreover linear by $(\ast)$ and the assumption on $X$.
We claim that 
\[
\delta B_{\hat{T}(X_\infty)}=\hat{T}(\delta B_{X_\infty})\subset L_\infty:=\{ [y_n]_n \in Y_\infty: y_n\in L\}\subset \hat{T}(X_\infty).
\]
Here $\delta>0$ is such that $\delta B_X$ is contained in the closed convex hull 
of $\ext B_X$. 
Thus the first inclusion follows from $(\ast)$ and the fact that 
$\sum_k \lambda_k T(a_k) \in L$ for every $a_k\in \ext B_X$ and 
$\lambda_k\geq0$ with $\sum_k\lambda_k=1$.
The second follows from the fact that if $y:=[y_n]_n\in L_\infty$, 
then for $x_n := n T^{-1}(\frac{1}{n}y_n)$, 
one has $[y_n]_n=\hat{T}([x_n]_n) \in\hat{T}(X_\infty)$. 
This claim implies that $L$ has non-empty interior in its linear span. 
Indeed, if $y\in L$ and $\lambda\in\IR$ are such that $\|\lambda y\|\le \delta$, 
then the constant sequence $y$ belongs to $L_\infty$ and so is 
$\lambda y$, which means that there is a sequence $(z_n)_n$ in $L$ 
such that $\| \lambda y - z_n\|\to0$ and so $\lambda y\in L$. 
Therefore, we can apply Mankiewicz's theorem (\cite{mankiewicz}) to $T$ and conclude 
that $T$ is affine. 
\end{proof}

\begin{proof}[Proof of Corollary~\ref{cor:man}]
It follows from the Russo--Dye theorem (Theorem I.8.4 in \cite{davidson} 
for the complex case and Theorem 7.2.4 in \cite{li} for the real case) 
that the closed unit ball of a unital complex $\mathrm{C}^*$-alge\-bra, 
as well as a real von Neumann algebra, coincides with the closed convex hull 
of its extreme points (unitary elements). 
Note that the real Russo--Dye theorem in \cite{li} states that $B_A$ is 
contained in the closed convex hull of $\{ \cos(h)e^k : h=h^*,k=-k^*\}$ 
and it is easily seen that $\cos(h)$ is contained in the closed convex hull of 
unitary elements in the real von Neumann algebra generated by $h$.
\end{proof}
\pagebreak[2]

The following beautiful example without the strong Mankiewicz property is provided 
for us by G. Schechtman (\cite{schechtman}).
\begin{example}\label{example}
Consider the set $K_0$ of all continuous strictly increasing functions $f$ from $[0,1]$ onto $[0,1]$ 
and put $K:=\overline{K_0}\subset L^1[0,1]$. 
Then, $K$ is compact and in particular it has empty interior. 
Let $T_0\colon K_0\to K_0$ be the map that sends $f$ to its inverse. 
Since the $L^1$ distance between two functions is the area enclosed 
by their graphs, $T$ is an isometry by the Fubini theorem.
The continuous extension $T\colon K\to K$ of $T_0$ is a surjective isometry 
which is not affine. 
Hence, the compact convex subset $K\subset L^1[0,1]$ does not have 
the strong Mankiewicz property.
\end{example}
\section{Proof of Theorem~\ref{thm:mau} for $\IB(\cH)$}
Since the proof of Theorem~\ref{thm:mau} for 
the type I factor $\IB(\cH)$, $\dim\cH>2$, is much simpler than the general $\mathrm{C}^*$-case, 
we give it here as an appetizer. 

\begin{lem}[cf.\ Lemma 2.1 in \cite{fw}]\label{lem:cond}
Let $T\colon S_X\to S_Y$ be a surjective isometry. 
Assume that there are $\{\vp_i\}_i\subset B_{X^*}$ 
and $\{\psi_i\}_i\subset B_{Y^*}$ such that 
$\vp_i=\psi_i\circ T$ and that the family $\{ \vp_i\}_i$ is norming for $X$. 
Then, $T$ extends to a linear isometry.
\end{lem}
\begin{proof}
By assumption, the linear map $U\colon X\to\ell_\infty$, defined by $(U x)(i)=\vp_i(x)$, 
is isometric. 
Since $\{\psi_i\}_i$ is also norming, the same holds true for $V\colon Y\to\ell_\infty$, 
which satisfies $U|_{S_X}=V\circ T$. 
Thus $T$ extends to a linear isometry (which is $V^{-1}\circ U$). 
\end{proof}

We note that any real linear functional $\vp$ on a complex Banach 
space $X_{\IC}$ has the complexification $\vp_{\IC}$ on $X_{\IC}$, 
which satisfies that  $\vp=\Re\vp_{\IC}$ and $\|\vp\|=\|\vp_{\IC}\|$ 
(see Lemma III.6.3 in \cite{conway}). 

\begin{lem}\label{lem:arens}
Let $\vp$ be a norm-one real or complex linear functional on 
a $\mathrm{C}^*$-alge\-bra $A$. 
If $a\in B_A$ is such that $\vp(a)=1$, then 
$\vp(x)=\vp(aa^*xa^*a)$ for all $x\in A$.
\end{lem}
\begin{proof}
This is a simple consequence of the Arens trick.
Since 
\[
\|(1-aa^*)x+a\| = \| \left[\begin{smallmatrix} (1-aa^*) & a \end{smallmatrix}\right]
\left[\begin{smallmatrix} x \\ 1 \end{smallmatrix}\right]\|
\le \| \left[\begin{smallmatrix} (1-aa^*) & a \end{smallmatrix}\right]\|\,
\|\left[\begin{smallmatrix} x \\ 1 \end{smallmatrix}\right]\|
\le \sqrt{1+\|x\|^2},
\]
one has
\[
| \lambda \vp((1-aa^*)x)+1 |^2 = |\vp(\lambda(1-aa^*)x+a)|^2 \le 1+|\lambda|^2\|x\|^2
\]
for all $\lambda$. This is possible only if $\vp((1-aa^*)x)=0$. 
The proof of the other side is similar. 
\end{proof}

We call a closed face $F\subset S_X$ an \emph{intersection face} if 
\[
F=\bigcap\{ E : \mbox{$E\subset S_X$ a maximal face containing $F$}\}.
\]
By Corollary 3.4 in \cite{tanaka2}, every norm-closed face $F$ of 
the unit sphere $S_A$ of a complex $\mathrm{C}^*$-alge\-bra $A$ is an intersection face. 
This fact persists for the real case. Indeed, if $A$ is a real $\mathrm{C}^*$-alge\-bra, 
then $G:=\{ x \in S_{A_{\IC}} : \frac{1}{2}(x+\cJ(x))\in F\}$ is an intersection face of 
$S_{A_{\IC}}$ by Corollary 3.4 in \cite{tanaka2} and so is $F=G\cap S_A$. 

\begin{lem}\label{lem:intersection}
Let $T\colon S_X\to S_Y$ be a surjective isometry 
and $F$ be an intersection face. 
Then, $T(F)$ is an intersection face such that $T(-F)=-T(F)$. 
In particular, there is $\psi\in B_{Y^*}$ such that $\psi=1$ on $T(F)$. 
\end{lem}
\begin{proof}
Recall from \cite{cd,tingley} and Proposition 2.3 in \cite{mori} that 
for any maximal face $E\subset S_X$, its image $T(E)\subset S_Y$ 
is a maximal face such that $T(-E)=-T(E)$. 
Thus $T(F)=T(\bigcap E)=\bigcap T(E)$ is an intersection face such that 
$T(-F)=T(\bigcap -E)=-\bigcap T(E)=-T(F)$. 
The last assertion follows from Eidelheit's separation 
theorem (see, e.g., Lemma 3.1 in \cite{tanaka2}).
\end{proof}

\begin{proof}[Proof of the Mazur--Ulam property for $\IB(\cH)$, $\dim\cH>2$]
Let $\xi,\eta\in\cH$ be unit vectors and $\vp(\,\cdot\,)=\ip{\,\cdot\,\xi,\eta}$ 
be the corresponding linear functional on $A:=\IB(\cH)$. 
We consider the maximal face 
\[
E_\vp:=\{ x\in S_A : \vp(x)=1\} = \{ x\in S_A : x\xi=\eta \}.
\]
By Lemma~\ref{lem:intersection}, there is $\psi\in B_{Y^*}$ 
such that $\psi=1$ on the face $T(E_\vp)$. 
We will show $\Re\vp=\psi\circ T$, which along with Lemma~\ref{lem:cond} 
proves Theorem~\ref{thm:mau} for $A=\IB(\cH)$. 

Let $v\in A$ denote the rank-one partial isometry such that $v\xi=\eta$. 
Let $u\in S_A$ be an arbitrary unitary element. 
Since $\dim\cH>2$, there is a sub-partial isometry $w$ of 
$u$ such that $w\perp v$, for example, pick a unit 
vector $\zeta \in \{\xi,u^*\eta\}^\perp$ and set $w=up_\zeta$, 
where $p_\zeta$ is the rank-one projection corresponding to $\zeta$.
Put  
\[
F(w):=\{ x\in S_A : xw^*w=w \} = \{ w+ y : y\in B_{(1-ww^*)A(1-w^*w)} \}
\]
the corresponding closed face, which contains $u$. 
Since  $T|_{F(w)}$ is affine by Lemma~\ref{lem:intersection} and Corollary~\ref{cor:man}, 
there is $\theta\in B_{A^*}$ such that 
$(\psi\circ T)(w+y)=(\psi\circ T)(w)+\Re\theta(y)$ for $y\in B_{(1-ww^*)A(1-w^*w)}$.
Since $w\pm v \in F(w)\cap (\pm E_{\vp})$, one has 
$\pm 1 = (\psi\circ T)(w\pm v)=(\psi\circ T)(w)\pm\Re\theta(v)$. 
This implies that $(\psi\circ T)(w)=0$ and $\theta(v)=1$, 
which means $\theta=\vp$ by Lemma~\ref{lem:arens} 
and the fact that $vv^*xv^*v=\vp(x)v$ for all $x\in A$.  
It follows that $(\psi\circ T)(u)=\Re\vp(u)$ for every unitary element $u$. 
Now let $w'$ be an arbitrary rank-one partial isometry. 
Since $T|_{F(w')}$ is affine and $F(w')$ is the closed convex hull 
of the unitary elements in $F(w')$, 
the previous result implies that $\psi\circ T=\Re\vp$ on $F(w')$.  
Since $\bigcup_{w'} F(w')$ is dense in $S_A$, we conclude 
by continuity that $\psi\circ T=\Re\vp$.
\end{proof}
\section{Convex combinations of a face and its opposite}
For any face $E\subset S_X$ and $\lambda\in[-1,1]$, put 
\[
E_\lambda := \{ x\in S_X :  \dist(x,E)\le1-\lambda\mbox{ and }\dist(x,-E)\le1+\lambda\}.
\]
Since $\dist(E,-E)=2$ for any convex subset $E\subset S_X$, 
the inequalities defining $E_\lambda$ are actually equalities. 
We upgrade Lemma~\ref{lem:intersection} as follows.

\begin{lem}\label{lem:tflambda}
Let $T\colon S_X\to S_Y$ be an surjective isometry and 
$F\subset S_X$ be an intersection face. 
Then, $T(F)$ is an intersection face such that $T(F_\lambda)=T(F)_\lambda$ 
for every $\lambda\in[-1,1]$.
\end{lem}
\begin{lem}\label{lem:tflconv}
For any face $E\subset S_X$, $\lambda_1,\lambda_2\in [-1,1]$, and $\alpha\in[0,1]$, 
one has 
\[
\bigl(\alpha E_{\lambda_1}+(1-\alpha)E_{\lambda_2}\bigr)\cap S_X
 \subset E_{\alpha\lambda_1+(1-\alpha)\lambda_2}.	
\]
\end{lem}
\begin{proof}
For given $x_i\in E_{\lambda_i}$,  
put $x_3:= \alpha x_1+(1-\alpha)x_2$ and $\lambda_3:=\alpha\lambda_1+(1-\alpha)\lambda_2$. 
For any $\ve>0$, take $y_i\in E$ such that 
$\| y_i - x_i \|\approx_{\ve} 1-\lambda_i$. 
Then, $\alpha y_1+(1-\alpha)y_2\in E$ and 
\[
\dist(x_3, E)\le \alpha\|x_1-y_1\|+(1-\alpha)\| x_2-y_2\|
\approx_\ve 1-\lambda_3.
\]
Since $\ve>0$ was arbitrary, 
one has $\dist(x_3, E)\le 1-\lambda_3$. 
The proof of the other inequality is similar and one has $x_3\in E_{\lambda_3}$.
\end{proof}
\pagebreak[2]

Combining Lemmas~\ref{lem:tflambda} and \ref{lem:tflconv}, we obtain the following. 
\begin{lem}\label{lem:tfl}
Let $T\colon S_X\to S_Y$ be a surjective isometry and $E\subset S_X$ 
be an intersection face. 
Then, for every $\lambda_1,\lambda_2\in [-1,1]$ and $\alpha\in[0,1]$, 
one has 
\[
\bigl( \alpha T(E_{\lambda_1})+(1-\alpha)T(E_{\lambda_2}) \bigr) \cap S_Y\subset T(E_{\alpha\lambda_1+(1-\alpha)\lambda_2}).	
\]
\end{lem}
\section{Facial structure of a $\mathrm{C}^*$-alge\-bra}
The following first three results are about Kadison's transitivity theorem. 
They are all well-known in the complex case. 

\begin{lem}[Kadison's transitivity theorem]\label{lem:kadison}
Let $A$ be a $\mathrm{C}^*$-alge\-bra and 
$p\in A^{**}$ be a finite rank projection.
Then, for any norm-one (resp.\ self-adjoint) element 
$x\in pA^{**}p$, there is a norm-one (resp.\ self-adjoint) 
element $a\in A$ such that $ap= x = pa$.
For any unitary element $x$ in $\pm\cU_0(pA^{**}p)$, the connected component of 
$\pm1$, there is a unitary element $a$ in $A$ (or the unitization of, if $A$ is not unital) 
such that $ap= x = pa$.
\end{lem}
\begin{proof}
In the complex case, this follows from Theorem II.4.15 in \cite{takesaki}.
We deal with the case of a real $\mathrm{C}^*$-alge\-bra $A$. 
Let $(A_{\IC},\cJ)$ be its complexification. 
Then for every norm-one (resp.\ self-adjoint) element $x\in pA^{**}p$, 
there is a norm-one (resp.\ self-adjoint) element $a'\in A_{\IC}$ 
such that $a'p= x = pa'$. Thus $a:=\frac{1}{2}(a'+\cJ(a'))\in A$ 
satisfies the desired condition. 
Now let $x\in \pm\cU_0(pA^{**}p)$ be a unitary element. 
Since $x=\pm x_1\cdots x_n$ for some $x_k\in\cU_0(pA^{**}p)$ 
with $\| x_k - p \|<2$, we may assume that $-1$ is not in the spectrum of $x$. 
Let $\log\exp(\sqrt{-1}\lambda)=\sqrt{-1}\lambda$ for $\lambda\in(-\pi,\pi)$. 
Then, $h:=\frac{1}{\sqrt{-1}}\log (x+p^\perp) \in A_{\IC}^{**}$ 
is a self-adjoint element such that $hp=ph$ and 
$\cJ(h)=-\frac{1}{\sqrt{-1}}\overline{\log} (x^*+p^\perp)=-h$.
Similarly as above, there is a self-adjoint element $b \in A_{\IC}$ 
such that $bp = php = pb$ and $\cJ(b)=-b$. 
It follows that $a=\exp(\sqrt{-1} b)$ is 
a unitary element in (the unitization of) $A$ 
such that $ap= x = pa$. 
\end{proof}

The assumption on the unitary element $x$ cannot be removed in general; 
E.g., there is no unitary element $f$ in $\{ f\in C([0,1],\IM_2(\IR)) : f(0)\in\IR 1\}$ 
such that $f(1)=[\begin{smallmatrix} 1 & 0 \\ 0 & -1 \end{smallmatrix}]$.
However, this is a rather special case. 
The finite dimensional real von Neumann algebra, $pA^{**}p$ in the above lemma, 
is a direct sum of $\IM_k(\IF)$'s, where $\IF\in\{\IR,\IC,\IH\}$. 
Among them, the unitary groups of $\IM_k(\IC)$ and $\IM_k(\IH)$ are connected, 
while the unitary group of $\IM_k(\IR)$ has two connected components 
according to the sign of the determinant. 
Thus, unless there is a nonzero central projection 
$z\le p$ in $A^{**}$ such that $zA^{**}\cong\IM_k(\IR)$,
one can inflate $p$ to a finite rank projection $p_0$ and $x$ to a 
unitary element in $\cU_0(p_0A^{**}p_0)$. 

Recall that a linear functional $\vp$ on a $\mathrm{C}^*$-alge\-bra $A$ 
is called a \emph{state} if it is positive and has norm one. 
It is said to be \emph{pure} if it is an extreme point of the state space. 

\begin{lem}\label{lem:pure}
Let $\vp$ be a pure state on a $\mathrm{C}^*$-alge\-bra $A$. 
Then, $p:=\supp(\vp)$ is a minimal projection in $A^{**}$. 
Let $L:=\{ a\in A : \vp(a^*a)=0\}$ be the corresponding left ideal and 
$(e_n)_n$ be an approximate unit for the $\mathrm{C}^*$-sub\-alge\-bra $L\cap L^*$.
Then, $e_n\to p^\perp$ ultrastrongly. 
In particular, any $\theta\in B_{A^*}$ such that $\lim_n\theta(1-e_n)=1$ 
coincides with $\vp$. 
\end{lem}
\begin{proof}
We view $\vp$ as a normal state on the second dual $A^{**}$.
Recall that $p:=\supp(\vp)$ is the smallest projection in $A^{**}$ such that 
$\vp(p)=1$. 
In the complex case, it is well-known that $p$ has rank one. 
In the real case, the set $\Omega$ of states $\vp_0$ on $A_{\IC}$ 
such that $\Re\vp_0=\vp$ on $A$ is a weak$^*$-closed face of the state 
space of $A_{\IC}$ and any pure state $\vp_0\in\ext\Omega$ satisfies 
$\vp_{\IC}=\frac{1}{2}(\vp_0+\bar{\vp}_0)$ on $A_{\IC}$, 
where $\bar{\vp}_0(a)=\overline{\vp_0(\cJ(a))}$ and 
$\vp_{\IC}$ is the complexification of $\vp$, 
which is the unique state extension of $\vp$ on 
$A_{\IC}$ such that $\vp_{\IC}=\bar{\vp}_{\IC}$.  
It follows that $p=\supp(\vp_{\IC})$ has rank at most two in $A_{\IC}^{**}$. 
Hence, by compactness, for any nonzero projection $q\le p$ 
there is $\ve>0$ such that $\ve\vp(q\,\cdot\,q)\le\vp(p\,\cdot\,p)$.
Since $\vp(p\,\cdot\,p)=\vp(\,\cdot\,)$ is a pure state, this implies $q=p$, 
proving that the projection $p$ is minimal.
We note that there is only one state on $pA^{**}p\cong\IR,\,\IC,\,\IH$. 
Since $e_ne_m \to e_m$ as $n\to\infty$, the ultrastrong limit 
$e:=\sup e_n\in A^{**}$ is a projection such that $e^\perp\geq p$.  
By Kadison's transitivity theorem (Lemma~\ref{lem:kadison}), one has 
\[
\{ x \in A : pxp=0 \} = L + L^* \subset \{ x\in A : e^\perp x e^\perp=0\}. 
\]
Indeed, for any $x\in A$ such that $pxp=0$, there is a finite rank 
projection $q\geq p$ in $A^{**}$ 
such that $px=pxq$ and there is $a\in A$ such that $aq = pxq = qa$.
Since $ap=pxp=0$ and $p(x-a)=pxq-pqa=0$, one has $a\in L$ and $x-a\in L^*$. 
The right inclusion is obvious. 
Hence the map $pAp \ni pxp\mapsto e^\perp x e^\perp$ is a well-defined 
continuous linear map. Since $p$ is of finite rank, this map is 
ultraweakly continuous on $pA^{**}p$ and so $e^\perp\le p$. 
This proves $e=p^\perp$. 
If $\theta\in B_{A^*}$ is such that $\lim_n\theta(1-e_n)=1$, 
then $\theta(e^\perp)=1$, which implies that $\theta$ is a state and 
$\theta=\vp$ by Lemma~\ref{lem:arens} 
and the uniqueness of the state on $e^\perp A^{**} e^\perp$.
\end{proof}
\pagebreak[2]

Recall that for any maximal face $E\subset S_X$ there is $\vp\in \ext B_{X^*}$
such that 
\[
E=E_\vp:=\{ x\in S_X : \vp(x)=1\}.
\]

\begin{lem}\label{lem:maxface}
Let $A$ be a $\mathrm{C}^*$-alge\-bra and $\vp\in \ext B_{A^*}$. 
Then, $|\vp|$ is a pure state and there is a unitary element $u$ 
in the unitization of $A$ such that $\vp(\,\cdot\,)=|\vp|(u^*\,\cdot\,)$ 
and $E_{\vp} =u E_{|\vp|}$.
\end{lem}
\begin{proof}
Every $\vp\in S_{A^*}$ has a polar decomposition 
$\vp(\,\cdot\,)=|\vp|(v^*\,\cdot\,)$ (see Section III.4 in \cite{takesaki}), 
where $v$ is a partial isometry in $A^{**}$. 
By Lemma~\ref{lem:arens}, 
one sees that the state $|\vp|$ is pure if (and only if) $\vp\in\ext B_{A^*}$. 
Hence $p:=v^*v$ and $q:=vv^*$ are minimal projections in $A^{**}$. 
By Lemma~\ref{lem:arens}, it suffices to show that there is a unitary 
element $u$ in (the unitization of) $A$ such that $up=v$. 

If $p=q$, then $v$ is a unitary element in $pA^{**}p\cong\IR,\IC,\IH$ 
and so there is a unitary element $u$ such that $up=v$ by Kadison's 
transitivity theorem (Lemma~\ref{lem:kadison}). 
From now on, we assume that $p\neq q$. 
It follows that $p\vee q - p \sim q - p\wedge q = q$ 
(Proposition V.1.6 in \cite{takesaki}) is a minimal projection 
and $(p\vee q - p)v(p\vee q - q)v^*(p\vee q - p) = \alpha(p\vee q - p)$ 
for some $\alpha>0$. 
Hence $w:=\alpha^{-1/2}(p\vee q - q)v^*(p\vee q - p)$ is a minimal partial 
isometry in $(p\vee q)A^{**}(p\vee q)$ such that $w\perp v$.
We claim that at least one of the two unitary elements $v \pm w$ 
in $(p\vee q)A^{**}(p\vee q)$ does not have both $\pm1$ in its spectrum. 
If this was not the case, then $v+w=e-e^\perp$ and $v-w=f-f^\perp$ for some 
minimal projections $e$ and $f$, but since $v=e+f-1$ is a partial isometry, 
this implies $e=f$ or $e\perp f$, which is absurd. 
Hence, by Kadison's transitivity theorem (see Proof of Lemma~\ref{lem:kadison}), 
there is a unitary element $u$ in (the unitization of) $A$ such that $up=v$. 
\end{proof}

\begin{lem}\label{lem:realvna}
Let $A$ be a von Neumann algebra 
and $E\subset S_A$ be a maximal face.
Then, $E$ coincides with the closed convex hull of unitary elements in $E$. 
Moreover, there is a net $(\Theta_n)_n$ of 
affine contractions from $E$ onto ultra\-weakly-closed face $E_n\subset E$ 
such that $\Theta_n\to\id_E$ in the point-norm topology 
and each $E_n$ is affinely isometrically isomorphic 
to the closed unit ball of a real von Neumann algebra. 
In particular, any surjective isometry $T\colon S_A\to S_Y$ is affine on $E$.
\end{lem}
\begin{proof}
By Lemma~\ref{lem:maxface}, we may assume that $E=E_\vp$ for some 
pure state $\vp$ on $A$. 
Let $L:=\{ a\in A : \vp(a^*a)=0\}$ and take an approximate unit 
$(e_n)_n$ for $L\cap L^*$. 
By enlarging the index set if necessary, we may find $\ve(n)>0$ 
such that $\ve(n)\to0$ and put $q_n := \chi_{(\ve(n),1]}(e_n)$. 
Then, $(q_n)_n$ is again an approximate unit for the hereditary 
$\mathrm{C}^*$-sub\-alge\-bra $L\cap L^*$ and so any 
$x\in E$ satisfies $1-x\in L\cap L^*$ and $q_n(1-x)\approx 1-x \approx (1-x)q_n$. 
This implies $(1-q_n)x \approx 1-q_n \approx x(1-q_n)$ 
and $q_n x\approx xq_n$. 
Therefore $\Theta_n(x) := q_n^\perp + q_n x q_n \in q_n^\perp+B_{q_nAq_n}$ 
satisfies $\Theta_n(x)\to x$ in norm. 
We put $E_n:= q_n^\perp+B_{q_nAq_n} \subset E_\vp$. 
Then, by the Russo--Dye theorem (Theorem 7.2.4 in \cite{li}) $E_n$ 
coincides with the closed convex hull of unitary elements in $E_n$. 
Also, $T|_{E_n}$ is affine by
Lemma~\ref{lem:intersection} and Corollary~\ref{cor:man}, and so is $T|_E$.
\end{proof}

This is enough for the proof of Theorem~\ref{thm:mau} for 
von Neumann algebras. More technical results below are needed 
to deal with $\mathrm{C}^*$-alge\-bras. 

Let $A$ be a $\mathrm{C}^*$-alge\-bra and $A^{**}$ be its second dual. 
A projection $p$ in $A^{**}$ is said to be \emph{compact} if it is the ultrastrong 
limit of an decreasing net of norm-one positive elements in $A$. 
(See \cite{ap} for more detail.)
For any nonzero compact projection $p\in A^{**}$, 
denote the corresponding face by 
\[
F(p) := \{ x\in S_A : xp = p = px \} = A\cap \{ p+y : y\in B_{p^\perp A^{**}p^\perp}\}
\]
and put for $\lambda\in[-1,1]$, 
\begin{align*}
F(p,\lambda) := \{ x\in S_A : xp=\lambda p = px \} 
 = S_A\cap\{ \lambda p+ y : y\in B_{p^\perp A^{**}p^\perp} \}.
\end{align*}

\begin{lem}\label{lem:fwc}
One has $\overline{F(p)}^{\sigma(A^{**},A^*)}=\{ p+y : y\in B_{p^\perp A^{**}p^\perp}\}$. 
\end{lem}
\begin{proof}
The inclusion $\subset$ is clear. 
For the other inclusion, take $y\in B_{p^\perp A^{**}p^\perp}$ arbitrary.
By Kaplansky's density theorem, there is a net $y_n\in B_A$ such that $y_n\to y$ ultrastrongly.
Since $p$ is compact, there is a net $p_n\in A$ such that $0\le p_n\le 1$ 
and $p_n\searrow p$. 
(We may assume that these nets are indexed by the same directed set.)
Then, $p_np=p=pp_n$ and 
hence the net 
\[
z_n:=p_n+(1-p_n)^{1/2}y_n(1-p_n)^{1/2} 
 =[\begin{smallmatrix} p_n^{1/2} & (1-p_n)^{1/2} \end{smallmatrix}]
 [\begin{smallmatrix} 1 & 0 \\ 0 & y_n\end{smallmatrix}]
 [\begin{smallmatrix} p_n^{1/2} \\ (1-p_n)^{1/2}\end{smallmatrix}]
\]
belongs to $F(p)$ (the latter expression shows $\|z_n\|\le1$). 
Since it converges ultrastrongly to $p+y$, we are done.
\end{proof}

\begin{lem}\label{lem:fvlambda}
One has $F(p)_\lambda=F(p,\lambda)$.
\end{lem}
\begin{proof}
We view $A^{**}\subset\IB(\cH)$. 
Let $x\in F(p)_\lambda$ be given 
and take a unit vector $\xi\in p\cH$.
For every $\ve>0$, there are $y\in F(p)$ and $z\in -F(p)$ 
such that $\| y-x \|\approx_\ve1-\lambda$ and $\|z-x\|\approx_\ve1+\lambda$.
Since one has $y\xi=\xi=-z\xi$ and 
\[
2\le\| \xi-x\xi\|+\|\xi+x\xi\|\le \|y-x\|+\|-z+x\|\approx_{2\ve} 2,
\]
$\xi-x\xi$ and $\xi+x\xi$ are almost parallel. Since $\ve>0$ was arbitrary, 
this means $x\xi$ is parallel to $\xi$ and $x\xi=\lambda \xi$, 
i.e., $xp=\lambda p$. The proof of $px=\lambda p$ is similar. 
This proves  $F(p)_\lambda\subset F(p,\lambda)$. 
For the converse inclusion, let $x=\lambda p+x'\in F(p,\lambda)$ be given. 
Then, $y := p+x'\in \overline{F(p)}^{\sigma(A^{**},A^*)}$ by Lemma~\ref{lem:fwc} and 
$\|x-y\|=1-\lambda$. 
By the Hahn--Banach separation theorem, one has 
\[
\dist(x,F(p))=\dist(x,\overline{F(p)}^{\sigma(A^{**},A^*)})\le 1-\lambda.
\]
The proof of the other inequality is similar and  $F(p)_\lambda\supset F(p,\lambda)$.
\end{proof}
\pagebreak[2]

The following is formally stronger than the Russo--Dye theorem 
(Theorem I.8.4 in \cite{davidson}), but it can be proved by adapting 
the standard proof.  
\begin{lem}\label{lem:rd}
Let $A$ be a unital complex $\mathrm{C}^*$-alge\-bra and 
$E\subset S_A$ be a maximal face. 
Then, $E$ coincides with the closed convex hull of unitary elements in $E$. 
\end{lem}
\begin{proof}
By Lemma~\ref{lem:maxface}, we may assume that 
$E=E_\vp$ for some pure state $\vp$ on $A$. 
Let $p:=\supp(\vp)$ and view elements in $A$ 
as operator valued $2 \times 2$ matrix 
in accordance with $p\oplus p^\perp$. 
Thus, $E=\{[\begin{smallmatrix} 1 & 0 \\ 0 & \ast \end{smallmatrix}]\}$.
Let $x\in E$ and $\ve>0$ be given. 
We set $v_0:=1$ and choose $u_k,v_k\in E\cap\cU(A)$ 
inductively as follows.
Fix $\delta>0$ very small and let 
$(1+\delta)v_{k-1}+x=w_k|(1+\delta)v_{k-1}+x|$
be the polar decomposition. Note that $w_k\in E\cap\cU(A)$, as it is easily 
seen from the operator valued matrix viewpoint. 
We define $u_k,v_k \in E\cap\cU(A)$ to be 
$w_k(|\frac{v_{k-1}+x}{2}|\pm\sqrt{-1}(1-|\frac{v_{k-1}+x}{2}|^2)^{1/2})$.
Then, 
\[
u_k+v_k=w_k|v_{k-1}+x|\approx_{\delta'}w_k|(1+\delta)v_{k-1}+x|
=(1+\delta)v_{k-1}+x\approx_{\delta}v_{k-1}+x,
\]
where $\delta'>0$ depends only on $\delta$ and converges to $0$ as 
$\delta$ converges to $0$. 
Thus, by choosing $\delta>0$ small enough, 
one has $v_{k-1}+x\approx_{\ve/3}u_k+v_k$. 
It follows that 
\[
v_0+nx \approx_{\ve/3} u_1+v_1+(n-1)x \approx_{\ve/3}\cdots\approx_{\ve/3} u_1+\cdots+u_n+v_n
\]
and $\| x-\frac{1}{n}\sum_{k=1}^n u_k \|<\ve$ for $n\geq3/\ve$. 
\end{proof}
\pagebreak[2]

We will need the following ad hoc result.
\begin{lem}\label{lem:adhoc}
Let $A$ be a unital complex $\mathrm{C}^*$-alge\-bra and 
$\psi\colon A\to\IC$ be a nonzero multiplicative linear functional. 
Then the closed unit ball of the real $\mathrm{C}^*$-alge\-bra 
$A^\psi_{\IR} := \psi^{-1}(\IR)$ has the strong Mankiewicz property.
\end{lem}
\begin{proof}
Let $\bar{A}=\{\bar{a}:a\in A\}$ denote the complex conjugate $\mathrm{C}^*$-alge\-bra 
of $A$. Then, $A^\psi_{\IR}$ is the real part of the complex $\mathrm{C}^*$-alge\-bra 
$A^\psi := \{ a\oplus \bar{b} \in A\oplus \bar{A} : \psi(a)=\overline{\psi(b)}\}$ 
with respect to the conjugate linear automorphism 
$\cJ(a\oplus \bar{b})=b\oplus\bar{a}$. 
By Theorem~\ref{thm:man}, it suffices to show that $B_{A^\psi_{\IR}}$ 
coincides with the closed convex hull of unitary elements in $A^\psi_{\IR}$. 
Let $x\in B_{A^\psi_{\IR}}$ and $\ve>0$ be given. 
We consider the second dual von Neumann algebra $A^{**}$ and 
the ultraweakly continuous extension $\psi\colon A^{**}\to\IC$. 
Then, $M:=\ker\psi$ is a von Neumann subalgebra 
such that $A^{**}=\IC\oplus M$ as a von Neumann algebra. 
Hence there are unitary elements $u_k$ 
in $(A^{**})^\psi_{\IR}=\IR\oplus M$ 
such that $x\approx_\ve\frac{1}{n}\sum_{k=1}^n u_k$. 
Let $h_k:=\frac{1}{\sqrt{-1}}\log u_k \in (A^{**})_{\mathrm{s.a.}}$,
where $\log \exp(\sqrt{-1}\lambda)=\sqrt{-1}\lambda$ for $\lambda\in[-\pi,\pi)$.
Since $\psi(u_k)\in\{1,-1\}$, 
one has $\psi(h_k)=\frac{1}{\sqrt{-1}}\log \psi(u_k)\in\{0,-\pi\}$. 
By Kaplansky's density theorem, there are bounded nets 
$(h_{k,i})_i$ in $A_{\mathrm{s.a.}}$ such that $h_{k,i}\to h_k$ ultrastrongly. 
We may assume that $\psi(h_{k,i})\in\{0,-\pi\}$. 
It follows that $u_{k,i}:=\exp(\sqrt{-1}h_{k,i})$ are unitary elements 
in $A^\psi_{\IR}$ such that 
$\frac{1}{n}\sum_{k=1}^n u_{k,i} \to \frac{1}{n}\sum_{k=1}^n u_k$ ultrastrongly. 
Hence, by the Hahn--Banach separation theorem, one has 
\smallskip

\hspace*{\fill}
$\dist(x, \conv\{ u_{k,i} : k,i\})
 = \dist(x, \overline{\conv\{ u_{k,i} : k,i\}}^{\sigma(A^{**},A^*)})
 < \ve.$
\hspace*{\fill}
\end{proof}
\pagebreak[2]

The following is the main technical result. 
\begin{prop}\label{prop:tech}
Let $A$ be a $\mathrm{C}^*$-alge\-bra, 
$T\colon S_A\to S_Y$ be a surjective isometry, and 
$E\subset S_A$ be a maximal face. 
Then, there is a net $(\Theta_n)_n$ of 
affine contractions from $E$ into closed convex subsets $K_n\subset E$ 
such that $\Theta_n\to\id_E$ in the point-norm topology, each $K_n$ is affinely isometrically isomorphic 
to the closed unit ball of a real $\mathrm{C}^*$-alge\-bra $C_n$, and 
each $T(K_n)$ is convex. 
In the case $A$ is a unital complex $\mathrm{C}^*$-alge\-bra, 
the above $C_n$ can be taken so that $B_{C_n}$ has the strong Mankiewicz 
property and so $T|_E$ is affine. 
\end{prop}
\begin{proof}
By Lemma~\ref{lem:maxface}, we may assume that $E=E_\vp$ for some 
pure state $\vp$ on $A$. 
We denote by $\tilde{A}$ the unitization of $A$ if $A$ is not unital, 
else $\tilde{A}=A$. 
We consider $\vp$ a pure state on $\tilde{A}$ and consider 
$L:=\{ a\in \tilde{A} : \vp(a^*a)=0\}$.

One can skip this paragraph if $A$ is unital. 
In case $A$ is not unital, let $\sigma$ denote 
the character on $\tilde{A}$ corresponding to the unitization. 
We claim that there is an approximate unit $(e_n)_n$ for $L\cap L^*$
such that $\sigma(e_n)=1$. 
Let $(e_n)_n$ be any approximate unit. 
Then, by Lemma~\ref{lem:pure}, 
we may assume $\lambda:=\inf\sigma(e_n)>0$. 
We consider the continuous function $h(t)=\min(\lambda^{-1}t,1)$. 
Then, $h(e_n)$ is an approximate unit such 
that $\sigma(h(e_n))=h(\sigma(e_n))=1$ for all $n$. 

Let $(e_n)_n$ be an approximate unit for $L\cap L^*$ such that 
$1-e_n\in A$ for all $n$ (which is equivalent to $\sigma(e_n)=1$).
By perturbation using functional calculus, we may assume that 
there is $f_n\in L\cap L^*$ such that $0\le e_n\le f_n\le1$ and 
$e_nf_n=e_n=f_ne_n$. 
(Although $(e_n)_n$ may not be increasing anymore, 
this does not matter for the following.)
Since $1-x \in L \cap L^*$ for every $x\in E_\vp$, 
one has $e_n(1-x)\approx 1-x \approx (1-x)e_n$.
This implies $(1-e_n)x \approx 1-e_n \approx x(1-e_n)$ 
and $e_n x\approx xe_n$. 
Therefore 
\[
x=((1-e_n)+e_n)x\approx (1-e_n) + e_n^{1/2} x e_n^{1/2} =: \Theta_n(x) \in E_\vp.
\]
See Proof of Lemma~\ref{lem:fwc} for the proof that $\Theta_n$ is contractive 
and $\|\Theta_n(x)\|\le1$ for every $x\in E_\vp$.

To ease notation, we fix $n$ and write $f:=f_n$. 
We consider $s:=1-f \in E_\vp$, its support projection $p:=\chi_{(0,1]}(s)\in A^{**}$, 
and the closed face
\[
F:=\{ x\in S_A : xs = s = sx\} = \{ x\in S_A : xp=p=px\} \subset E_\vp
\]
(although $p$ is not a compact projection).
Since $(1-e_n)s=s=s(1-e_n)$, one has $\Theta_n(E_\vp)\subset F$.
If $p\in A$, then the face $K_n:=F=p+B_{p^\perp A p^\perp}$ satisfies 
the desired property with $C_n=p^\perp A p^\perp$. 
Let's assume $p\notin A$ and 
consider the $\mathrm{C}^*$-sub\-alge\-bra 
\[
C:=\{ a\in A : ap=\gamma p = pa \mbox{ for some scalar }\gamma\}.
\]
For the following, it is probably easier to digest if one views 
elements in $A$ as operator valued $2\times2$ matrices 
in accordance with $p\oplus p^{\perp}$. 
Thus 
$s=[\begin{smallmatrix} s & 0 \\ 0 & 0\end{smallmatrix}]$ (slightly abusing the notation), 
$p=[\begin{smallmatrix} 1 & 0 \\ 0 & 0\end{smallmatrix}]$,  
$F=\{ [\begin{smallmatrix} 1 & 0 \\ 0 & \ast\end{smallmatrix}]\}$,
and $C=\{ [\begin{smallmatrix} \gamma & 0 \\ 0 & \ast\end{smallmatrix}]\}$.
Since $p\notin A$, one has $\|px\|\le\|p^\perp x\|$ for all $x\in C$. 
Hence, there is a norm-one (multiplicative) linear functional $\psi$ 
on $p^\perp C$ such that $\psi(p^\perp x)p=px$ for all $x\in C$, 
or equivalently 
$x=[\begin{smallmatrix} \psi(y) & 0 \\ 0 & y\end{smallmatrix}]$ 
for all $x\in C$ and $y:=p^\perp x$. 
Hence, $\| x \| = \|(1-s)x\|$ for all $x\in C$ and 
one has 
\[
F\subset K_n := s+(1-s)B_{C^\psi_{\IR}}
 = S_A \cap \{ [\begin{smallmatrix} s+\gamma (1-s) & 0 \\ 0 & \ast \end{smallmatrix}] : 
 \gamma\in[-1,1]\} \subset E_\vp.
\] 
By Lemma~\ref{lem:adhoc}, the closed unit ball $B_{C^\psi_{\IR}}$ satisfies the strong Mankiewicz property 
provided that $A$ (and hence $C$) is a unital complex $\mathrm{C}^*$-alge\-bra. 
It is left to show that $T(K_n)$ is convex. 
For $\gamma\in[-1,1]$, put $h_\gamma(\lambda):=\lambda+\gamma(1-\lambda)$. 
For $i=1,2$, put 
\[
G^i_m(\gamma):=E_\vp\cap{\bigcap}_k\,
   F\bigl(\chi_{[\frac{2k-2+i}{2m},\frac{2k-1+i}{2m}]}(s),h_\gamma(\frac{k}{m})\bigr)
\mbox{ and }H^i_m(\gamma) := E_\vp\cap\cN_{\frac{1}{m}}(G^i_m(\gamma)),
\]
where the intersection is over $k=1,2,\ldots,m$ for 
which $\chi_{[\frac{2k-2+i}{2m},\frac{2k-1+i}{2m}]}(s)\neq0$ and 
$\cN_\delta$ means the $\delta$-neighborhood in $A$. 
Note that $\chi_{[\frac{2k-2+i}{2m},\frac{2k-1+i}{2m}]}(s)\le p$ for all $i$, $m$, and $k$.
By Lemmas~\ref{lem:fvlambda} and \ref{lem:tflconv}, one has 
\[
\alpha G^i_m(\gamma_1)+(1-\alpha) G^i_m(\gamma_2) \subset G^i_m(\gamma_3)
\mbox{ and }
\alpha H^i_m(\gamma_1)+(1-\alpha) H^i_m(\gamma_2) \subset H^i_m(\gamma_3)
\]
for every $\gamma_1,\gamma_2\in[-1,1]$, $\alpha\in[0,1]$, and 
$\gamma_3:=\alpha\gamma_1+(1-\alpha)\gamma_2$. 
We claim that 
\begin{align*}
K(\gamma):=\{ x\in E_\vp : px=h_\gamma(s)=xp\}
 = {\bigcap}_{m\in\IN}\, (H^1_m(\gamma)\cap H^2_m(\gamma)).
\end{align*}
To prove the inclusion $\subset$, we define $g_m$ to be 
the continuous function such that 
$g_m(0)=\gamma$, $g_m(\lambda)=h_\gamma(\frac{k}{m})$ for 
$\lambda\in[\frac{2k-1}{2m},\frac{2k}{2m}]$, and linear on $[\frac{2k-2}{2m},\frac{2k-1}{2m}]$ for $k=1,\ldots,m$.
Then, $\|g_m-h_\gamma\|_\infty\le \frac{1}{m}$ and $(g_m-h_\gamma)(0)=0$.
It follows that $(g_m-h_\gamma)(s)\in A\cap pAp$ and for any $x\in K(\gamma)$, 
one has $x+(g_m-h_\gamma)(s)\in G^1_m(\gamma)$.
This proves $x\in H^1_m(\gamma)$. 
The proof of $x\in H^2_m(\gamma)$ is similar.
For the converse inclusion, take $x$ from the RHS of the claimed equality. 
Since $x\in H^1_m(\gamma)\cap H^2_m(\gamma)$, there are 
$y^i_m\in G^i_m(\gamma)$ such that $\| x - y^i_m \|\le\frac{1}{m}$. 
For the projection $p^i_m:=\sum_{k=1}^m \chi_{[\frac{2k-2+i}{2m},\frac{2k-1+i}{2m}]}(s)$ 
in $A^{**}$, one has 
$\|h_\gamma(s)p^i_m - y^i_m p^i_m\|\le\frac{1}{m}$.
Hence, $\| (h_\gamma(s) - x)(p^1_m\vee p^2_m)\|\le\frac{2}{m}$.
Since $p^1_m\vee p^2_m\to p$ ultrastrongly, one sees 
$h_\gamma(s)=xp$. The proof of $h_\gamma(s)=px$ is similar. 
Now, since $K_n=\bigcup_{\gamma\in[-1,1]} K(\gamma)$ and 
\[
T^{-1}(\alpha T(K(\gamma_1)) + (1-\alpha)T(K(\gamma_2)))
\subset {\bigcap}_{m\in\IN}\, (H^1_m(\gamma_3)\cap H^2_m(\gamma_3))=K(\gamma_3)
\]
by Lemma~\ref{lem:tfl}, one concludes that $T(K_n)$ is convex. 
\end{proof}
\section{Proof of Theorem~\ref{thm:mau}} 
We first give the proof of Theorem~\ref{thm:mau} for the case where 
$A$ is not a type $\mathrm{I}_k$ factor with $k=1,2$. 
The $\mathrm{I}_1$ factors $\IR,\,\IC,\,\IH$ are real Hilbert spaces and 
the Mazur--Ulam property for them is already known (see \cite{day,cd}). 
The case of $\mathrm{I}_2$ factor is dealt with separately. 

\begin{proof}[Proof of Theorem~\ref{thm:mau} when $A$ is not a a type $\mathrm{I}_k$ factor with $k=1,2$]
We will show that for any surjective isometry $T\colon S_A\to S_Y$ and 
any pure state $\vp$ on $A$, there is $\psi\in B_{Y^*}$ such that $\Re\vp=\psi\circ T$. 
This yields the assertion by Lemmas~\ref{lem:cond} and \ref{lem:maxface}.
Let $p:=\supp(\vp) \in A^{**}$ and consider the maximal face 
\[
E_\vp:=\{ x\in S_A : \vp(x)=1\} = \{ x\in S_A : xp = p = px \}.
\]
By Lemma~\ref{lem:intersection}, 
there is $\psi\in B_{Y^*}$ such that $\psi=1$ on the face $T(E_\vp)$. 
We will show $\Re\vp=\psi\circ T$.

Let $u\in S_A$ be an arbitrary unitary element. 
By the assumption that $A$ is not a type $\mathrm{I}_k$ factor with $k=1, 2$, 
there is a minimal projection 
$q\in A^{**}$ such that $q\perp p\vee u^*pu$.
We consider the face 
\[
F:=\{ x\in S_A :  xq = uq \} = S_A\cap (uq + B_{(uqu^*)^\perp Aq^\perp}).
\]
Since $T|_F$ is affine by Proposition~\ref{prop:tech} or Lemma~\ref{lem:realvna}, 
there are $\gamma\in\IR$ 
and $\theta\in B_{A^*}$ such that 
$(\psi\circ T)(x)=\gamma+\Re\theta(x)$ for $x\in F$.
By Kadison's transitivity theorem (Lemma~\ref{lem:kadison}), 
there is $x_0\in B_A$ such that $x_0q=uq$ 
and $x_0p=0=px_0$. 
Then, $x_0\in F\cap L\cap L^*$, where $L:=\{ a \in A : \vp(a^*a)=0 \}$. 
Let $(e_n)_n$ be an approximate unit for the 
$\mathrm{C}^*$-sub\-alge\-bra $L\cap L^*$ 
such that $e_n\geq |x_0|$ for all $n$. 
Since $e_n\geq |x_0|$, one has $e_n q = q$ and 
$x_0 \pm (1-e_n) \in F\cap (\pm E_\vp)$ for all $n$.
Thus 
\[
\pm 1 = (\psi\circ T)(x_0\pm(1-e_n)) = (\psi\circ T)(x_0) \pm \theta(1-e_n).
\]
This implies that $(\psi\circ T)(x_0)=0$ and $\theta(1-e_n)=1$ for all $n$.
By Lemma~\ref{lem:pure}, 
one sees $\theta=\vp$ and so $(\psi\circ T)(u)=\Re\vp(u)$. 
Now $\Re\vp=\psi\circ T$ follows from Lemma~\ref{lem:rd} 
or \ref{lem:realvna}.
\end{proof}

For the rest of the paper, we put $A=\IM_2(\IF)$, where $\IF\in\{\IR,\IC,\IH\}$. 
\begin{lem}\label{lem:asab}
Let $\tr$ denote the tracial state on $A$ and put $\cH:=A_{\mathrm{sa}}\cap\ker\tr$.
Then, $\cH$ is a real Hilbert space and 
the self-adjoint part $A_{\mathrm{sa}}$ of $A$ is isometrically isomorphic to 
the $\ell_1$-direct sum $\IR1\oplus_1\cH$, 
via $A_{\mathrm{sa}} \ni a \mapsto \tr(a)1 \oplus (a-\tr(a)1)$. 
Moreover, if $x\in S_A$ satisfies $\|1\pm x\|=2$, then $x\in\cH$.
\end{lem}
\begin{proof}
Every $b \in \cH$ is of the form $\mu p - \mu p^{\perp}$ 
for some $\mu\in\IR$ and a minimal projection $p$. 
It follows that $\cH$ is a real Hilbert space 
as $\| b \|=\tr(b^*b)^{1/2}$ and that $a=\lambda1+b$ satisfies 
\[
\|a\|=\max\{|\lambda+\mu|,|\lambda-\mu|\}=|\lambda|+|\mu|=|\tr(a)| + \| a-\tr(a)1\|. 
\]
Now suppose $x\in S_A$ satisfies $\|1\pm x\|=2$. Then, there are unit vectors 
$\xi_+$ and $\xi_-$ such that $x\xi_{\pm}=\pm\xi_{\pm}$. 
It follows that the minimal projections $p_{\pm}$ that satisfy $p_{\pm}\xi_{\pm}=\xi_{\pm}$ 
satisfies $p_{\pm}xp_{\pm}=\pm p_{\pm}$. Since $A=\IM_2(\IF)$, 
one concludes that $x=p_+-p_- \in \cH$. 
\end{proof}

\begin{lem}\label{lem:asaext}
Let $T\colon S_A\to S_Y$ be any surjective isometry. 
Then, $T|_{S_A\cap A_{\mathrm{sa}}}$ admits a linear extension.
\end{lem}
\begin{proof}
Let $\tilde{T}\colon A\to Y$ denote the homogeneous extension of $T$, 
which is given by $\tilde{T}(a)=\|a\|T(\frac{a}{\|a\|})$ for $a\neq0$ and $\tilde{T}(0)=0$. 
For any $b\in S_{\cH}$, the convex hull of $1$ and $b$ is contained in 
$S_A$ by Lemma~\ref{lem:asab} and hence $T$ is affine there. 
Moreover, since $T$ preserves antipodal points (\cite{tingley}), $\tilde{T}$ is linear 
on the linear span of $1$ and $b$.
It remains to show that $\tilde{T}$ is linear on $\cH$. 
For this, we first prove that $\|T(b)+T(c)\| = \| b + c \|$ 
and $T(\frac{1}{\|b+c\|}(b+c))=\frac{1}{\|b+c\|}(T(b)+T(c))$ for every $b,c\in S_\cH$. 

Let $b,c\in S_\cH$ be given. 
We may assume that they are not parallel. 
For any $\lambda\in[-1,1]$, 
one has 
\begin{align*}
\| 2(1-|\lambda|) 1 +\lambda (b + c) \|
 &= \|((1-|\lambda|)1 + \lambda b) - (-(1-|\lambda|)1 - \lambda c) \| \\
 &= \|T((1-|\lambda|)1 + \lambda b) - T(-(1-|\lambda|)1 - \lambda c) \| \\
 &= \| 2(1-|\lambda|) T(1) +\lambda (T(b) + T(c)) \|.
\end{align*}
In particular, 
$\mu:=\|b + c \| = \| T(b) + T(c) \|>0$ and put $\lambda:=\frac{2}{2+\mu}$. Then, 
\begin{align*}
\| 1 \pm T^{-1}(\frac{1}{\mu}(T(b) + T(c))) \|
 &=\frac{1}{\lambda\mu}\|  2(1-\lambda)T(1) \pm \lambda( T(b) + T(c))\| \\
 &=\| 1 \pm \frac{1}{\mu}(b + c) \| = 2.
\end{align*}
By Lemma~\ref{lem:asab}, this implies that $x:=T^{-1}(\frac{1}{\mu}(T(b) + T(c))) \in S_{\cH}$. 
If $\mu\le 1$, then  
\begin{align*}
1-\mu+\|\mu x - b \| &= \|( (1-\mu)1 + \mu x ) - b \| \\
 &= \| (1-\mu)T(1)+\mu T(x) - T(b) \| 
=1-\mu + \|T(c)\|
\end{align*}
and hence $\|\mu x- b \|=1$. 
On the other hand, if $\mu\geq 1$, then  
\begin{align*}
1-\frac{1}{\mu}+\|x - \frac{1}{\mu}b \| = \| x - ((1-\frac{1}{\mu})1 + \frac{1}{\mu}b) \| 
 = 1-\frac{1}{\mu}+ \frac{1}{\mu}\|T(c)\|
\end{align*}
and hence $\|\mu x - b \|=1$ again. 
Thus in any case, one has $\|\mu x - b \|=1$, and similarly $\|\mu x - c \|=1$. 
Also $\|\mu x - 0\|=\mu$. 
Therefore by trilateration, one has $\mu x=b+c$.

It follows that 
\[
\tilde{T}(\frac{\| c \| }{ \|b\|+\|c\|}b+ \frac{\|b \|}{ \|b\|+\|c\|}c)
=\frac{\| c \|}{ \|b\|+\|c\|}\tilde{T}(b) + \frac{\|b \|}{ \|b\|+\|c\|}\tilde{T}(c)
\]
for every $b,c\in\cH$. 
Thus the bisection method and continuity of $\tilde{T}$ imply that 
$\tilde{T}$ is affine on the segment $[b_0,c_0]$ for any $b_0,c_0\in \cH$.
Note that the bisection process works 
whenever $b_0$ and $c_0$ are not parallel, because in which case 
the norm on $[b_0,c_0]$ is bounded above and away from zero. 
This proves that $\tilde{T}$ is linear on $\cH$.
\end{proof}

We consider the unitary group $\cU$ of $A$ and the diagonal subgroup 
\[
\cD=\{[\begin{smallmatrix} \alpha & \\ & \beta \end{smallmatrix}] : \alpha,\beta \in\IF,\ |\alpha|=1=|\beta|\}\subset\cU.
\]
We claim that 
\[
\cU = \cD\cdot\{[\begin{smallmatrix} \lambda & \sqrt{1-\lambda^2} \\ \sqrt{1-\lambda^2} & -\lambda \end{smallmatrix}] : \lambda\in[0,1]\}\cdot\cD.
\]
Indeed, for any $u\in\cU$, it is obvious that 
$\cD\cdot u \ni [\begin{smallmatrix} \lambda & \gamma \\ \delta & -\mu \end{smallmatrix}]$ 
for some $\lambda,\mu\in[0,1]$. 
But since the latter is unitary, one must have $\lambda=\mu$ and $|\gamma|=|\delta|$.
If $\lambda=\mu>0$, then one moreover has $\gamma=\delta^*$ and so by conjugating 
$[\begin{smallmatrix} 1 &  \\  & \delta \end{smallmatrix}]$, one obtains the claim. 
One obtains the claim in the case $\lambda=\mu=0$ also. 

\begin{lem}\label{lem:2by2}
Let $E=\{ [\begin{smallmatrix} 1 &  \\  & \ast \end{smallmatrix}]\}$ be the maximal face 
of $S_A$ and $\vp_0$ be the corresponding pure state such that $\vp_0=1$ on $E$. 
Let $T\colon S_A\to S_Y$ be a surjective isometry and 
$\psi\in S_{Y^*}$ be such that $\psi\circ T=1$ on $E$. 
Then, one has 
$\psi\circ T=\vp_0$ on $S_A\cap A_{\mathrm{sa}}$ as well as on $\cD$.
\end{lem}
\begin{proof}
By Lemma~\ref{lem:asaext}, the map $\vp:=\psi\circ T$ admits a linear 
extension on $A_{\mathrm{sa}}$. 
Since $A_{\mathrm{sa}}\cong \IR\oplus_1\cH$ by Lemma~\ref{lem:asab}, 
the linear functional $\vp$ corresponds to a norm one element in $\IR\oplus_\infty\cH$, 
which is easily seen to be $1\oplus  [\begin{smallmatrix} 1 &  \\  & -1 \end{smallmatrix}]$. 
It follows that $\vp(a)=\tr(a)+\tr([\begin{smallmatrix} 1 &  \\  & -1 \end{smallmatrix}]a)=\vp_0(a)$ 
for $a\in S_A\cap A_{\mathrm{sa}}$. 
Next, let $\beta\in\IF$ be such that $|\beta|=1$. Then, $\vp$ 
is affine on $\{ [\begin{smallmatrix} \ast &  \\  & \beta \end{smallmatrix}]\}$. 
Since $\vp$ is a contractive map such that 
$\vp( [\begin{smallmatrix} \pm 1 &  \\  & \beta \end{smallmatrix}])=\pm1$, 
one obtains $\vp( [\begin{smallmatrix} \alpha &  \\  & \beta \end{smallmatrix}])=
\Re\alpha$.
\end{proof}

\begin{proof}[Proof of Theorem~\ref{thm:mau} for $A=\IM_2(\IF)$]
Let $T\colon S_A\to S_Y$ be given and $\psi\in S_{Y^*}$ be given such that $\vp:=\psi\circ T$ satisfies $\vp=\vp_0$ on $E$. 
It suffices to show $\vp=\vp_0$ everywhere. 
Let $x\in S_A$ be given and consider the polar decomposition $x=v|x|$. 
We may assume that $v$ is a unitary element. Since the surjective isometry 
$T(v\,\cdot\,)$ admits a linear extension on $A_{\mathrm{sa}}$ 
by Lemma~\ref{lem:asaext} and $\ext A_{\mathrm{sa}}\subset\cU$, 
to prove $\vp(x)=\vp_0(x)$, it suffices to show $\vp=\vp_0$ on $\cU$. 
Let $u\in\cU$ be given and write it as $u=[\begin{smallmatrix} \alpha_1 & \\ & \alpha_2 \end{smallmatrix}][\begin{smallmatrix} \lambda & \sqrt{1-\lambda^2} \\ \sqrt{1-\lambda^2} & -\lambda \end{smallmatrix}][\begin{smallmatrix} \beta_1 & \\ & \beta_2 \end{smallmatrix}]=:axb$. 
Since the surjective isometry $S_A\ni z\mapsto T(azb) \in S_Y$ admits a linear 
extension on $A_{\mathrm{sa}}$, 
one has 
\[
\vp(u)=\vp(axb)=\lambda\vp(a[\begin{smallmatrix} 1 &  \\  & -1\end{smallmatrix}]b)+
 \sqrt{1-\lambda^2}\vp(a[\begin{smallmatrix}  & 1 \\ 1 & \end{smallmatrix}]b).
\]
Since $a[\begin{smallmatrix} 1 &  \\  & -1\end{smallmatrix}]b\in\cD$, 
one has $\vp(a[\begin{smallmatrix} 1 &  \\  & -1\end{smallmatrix}]b)=\vp_0(a[\begin{smallmatrix} 1 &  \\  & -1\end{smallmatrix}]b)$ by Lemma~\ref{lem:2by2}.
On the other hand, since
$a[\begin{smallmatrix}  & 1 \\ 1 & \end{smallmatrix}]b=
[\begin{smallmatrix} 1 & \\ & \alpha_2\beta_1 \end{smallmatrix}][\begin{smallmatrix}  & 1 \\ 1 & \end{smallmatrix}][\begin{smallmatrix} 1 & \\ & \alpha_1\beta_2 \end{smallmatrix}]$
and $T'(\,\cdot\,)=T([\begin{smallmatrix} 1 & \\ & \alpha_2\beta_1 \end{smallmatrix}]\,\cdot\,[\begin{smallmatrix} 1 & \\ & \alpha_1\beta_2 \end{smallmatrix}])$ is a surjective isometry such 
that $\vp':=\psi\circ T'$ satisfies $\vp'=\vp_0$ on $E$, Lemma~\ref{lem:2by2} implies that 
$\vp(a[\begin{smallmatrix}  & 1 \\ 1 & \end{smallmatrix}]b)
=\vp'([\begin{smallmatrix}  & 1 \\ 1 & \end{smallmatrix}])=\vp_0([\begin{smallmatrix}  & 1 \\ 1 & \end{smallmatrix}])=0$.
These two imply $\vp(u)=\vp_0(u)$.
\end{proof}

This finishes the proof of Theorem~\ref{thm:mau}. 
We note that a similar proof yields the Mazur--Ulam property 
for the self-adjoint part $A_{\mathrm{sa}}$ of any real von Neumann 
algebra $A$.

\end{document}